\documentclass[12pt]{amsart}

\usepackage{amsmath}
\usepackage{amsthm}
\usepackage{amssymb}

\usepackage{epstopdf}

\usepackage{fullpage}

\usepackage{geometry}                % See geometry.pdf to learn the layout options. There are lots.
\usepackage{graphicx}

\usepackage{mathrsfs}
\usepackage{multicol}

\usepackage{tikz}

\usepackage{url}

\newtheorem{theorem}{Theorem}
\newtheorem{corollary}[theorem]{Corollary}
\newtheorem{lemma}[theorem]{Lemma}

\begin{document}

\title{Avoidability index for binary patterns with reversal}

\author{James D. Currie and Philip Lafrance}

\address{Department of Mathematics and Statistics \\
University of Winnipeg \\
515 Portage Avenue \\
Winnipeg, Manitoba R3B 2E9 (Canada)}

\email{j.currie@uwinnipeg.ca, plafranc@uwaterloo.ca}

\thanks{The first author was supported by an NSERC Discovery Grant, and also by Deutsche Forschungsgemeinschaft, which supported him through its Mercator program. The second author was supported through the NSERC USRA program.}

\subjclass[2000]{68R15}

\date{\today}

\begin{abstract}
For every pattern  $p$  over the alphabet $\{x,y,x^R,y^R\}$, we specify the least $k$ such that $p$ is $k$-avoidable.
\end{abstract}

\maketitle

\section{Introduction}
The study of words avoiding patterns is a major theme in combinatorics on words, explored by Thue and others \cite{thue,BEM,zimi,roth,cass,loth97,loth02}. The reversal map is a basic notion in combinatorics on words, and it is therefore natural that recently work has been done on patterns with reversals by Shallit and others \cite{CR,DMSS, RS}. (More general ideas, such as patterns with involutions or other permutations, have also been studied very recently by the first author and others \cite{MMN,BN,curr,BCN}.) Shallit et al. \cite{DMSS} recently asked whether the number of binary words avoiding $xxx^R$ grows polynomially with length, or exponentially. The surprising answer by Currie and Rampersad \cite{CR} is `Neither'. As B. Adamczewski \cite{adam}  has observed, this implies that the language of binary words avoiding $xxx^R$ is not context-free -- a result which has so far resisted proof by standard methods.

Basic questions about patterns with reversal have not yet been addressed. In this article, we completely characterize the $k$-avoidability of an arbitrary  binary pattern with reversal. This is a direct (and natural) generalization of the work of Cassaigne \cite{cass} characterizing $k$-avoidability for binary patterns without reversal, and involves a blend of classical results and new constructions. 

\section{Preliminaries}
For general concepts and notations involving combinatorics on words, we refer the reader to the work of Lothaire \cite{loth97,loth02}.
	Let $\Sigma$ be the alphabet $\Sigma=\{x,x^R,y,y^R\}$. We call a word $p\in\Sigma^*$ a {\bf binary pattern with reversal}. For a positive integer $k$, let $T_k$ be the alphabet $\{0,1,\ldots, k-1\}$. We refer to words (resp., sequences, morphisms) on $T_2$ as {\bf binary words} (resp., {\bf binary sequences, binary morphisms}). For words of $T_k^*$, let $\_^R$ denote the reversal antimorphism on $T_k$; thus if $a_1$, $a_2,\ldots, a_n\in T_k$, then
$$(a_1a_2\cdots a_n)^R=a_na_{n-1}\cdots a_1.$$
(Note the two distinct usages of $\_^R$: In $\Sigma$, the notation distinguishes pairs of alphabet letters; on $T_k^*$ it stands for reversal.)
We say that a morphism $f:\Sigma^*\rightarrow T_k^*$ {\bf respects reversal} if $f(x^R)=f(x)^R$, $f(y^R)=f(y)^R$. Thus any morphism from $\{x,y\}^*$ to $T_k^*$  extends uniquely to a morphism on $\Sigma$ respecting reversal. Let $p$ be a binary pattern with reversal. An {\bf instance} of $p$ is the image of $p$ under some non-erasing morphism which respects reversal. For example,  an instance of $p=xyyx^R$ is a word $XYYX^R$, where $X$ and $Y$ are non-empty; this is the image of $p$ under the non-erasing morphism respecting reversal given by $f(x)=X$, $f(y)=Y$. If pattern with reversal $p$ does not contain either of $x^R$ and $y^R$, then an instance of $p$ is simply an instance of pattern $p$ in the usual sense.

Let $k$ be a positive integer. Let $p$ be a  binary pattern with reversal. A word $w$ {\bf avoids} $p$ if no factor of $w$ is an instance of $p$. Pattern $p$ is $k$-avoidable if there are arbitrarily long words of $T_k^*$ which avoid $p$; equivalently, there is an $\omega$-word ${\bf w}$ over $T_k$ such that every finite prefix of ${\bf w}$ avoids $p$. If $p$ is not $k$-avoidable, it is $k$-unavoidable; note that every factor of a $k$-unavoidable word is $k$-unavoidable. Word $p$ is {\bf avoidable} if it is $k$-avoidable for some $k$; otherwise, $p$ is {\bf unavoidable}. If $p$ is avoidable, then the {\bf avoidability index} of $p$ is defined to be the least $k$ such that $p$ is $k$-avoidable. If $p$ is unavoidable, we define the unavoidability index of $p$ to be $\infty$.

\section{Classification}

Consider the morphisms $\iota_1, \iota_2$ on $\Sigma^*$ given by:
$$\iota_1(x)=x^R,\iota_1(x^R)=x, \iota_1(y)=y,\iota_1(y^R)=y^R,$$ $$\iota_2(x)=y,\iota_2(x^R)=y^R,\iota_2(y)=x,\iota_2(y^R)=x^R.$$ Thus $\iota_1$ switches $x$ and $x^R$, while $\iota_2$ switches $x$ and $y$, $x^R$ and $y^R$. Thus $\iota_2(\iota_1(\iota_2))$ switches $y$ and $y^R$. One checks the following:
\begin{lemma}If $f:\Sigma^*\rightarrow T_k^*$ is a morphism respecting reversal, then so is $f\circ \iota_j$ for $j=1,2$.
\end{lemma}

 Let $\iota_3$ denote the reversal antimorphism on $\Sigma^*$.
\begin{lemma}
Let $p$ be a binary pattern with reversal. If $w$ is an instance of $p$, then $w^R$ is an instance of $\iota_3(p)$.
\end{lemma}
For $j=1,2,3$, $\iota_j^2$ is the identity morphism on $\Sigma^*$. It follows that the relation on $\Sigma^*$ given by
$$p\sim q\iff \mbox{$q$ is obtained from $p$ by a sequence of applications of $\iota_1$, $\iota_2$ and $\iota_3$}$$

is an equivalence relation.
Combining the previous two lemmas gives the following:

\begin{lemma} Let $k$ be a positive integer.  Let $p$, $q$ be binary patterns with reversal. Suppose that $q\sim p$. Then $p$ is $k$-avoidable if and only if $q$ is $k$-avoidable.
\end{lemma}

Consider the lexicographic order on $\Sigma^*$ generated by $x<x^R<y<y^R$. If $p\in\Sigma^*$, define
$\ell(p)$ to be the lexicographically least element of the equivalence class of $p$ under $\sim.$ For example, $\ell(x^Ryy)=xxy$.

 Let
\begin{eqnarray*}
C_1(p)&=&\{p\},\\
C_{n+1}(p)&=&C_n(p)\cup \bigcup_{j=1}^3\iota_j(C_n(p)), n\ge 1.
\end{eqnarray*}
Since the $\iota_j$ preserve length, for any pattern $p$ only the finitely many words of $\Sigma^{|p|}$ can be equivalent to $p$. Thus for some positive integer $m$ we will have $C_{m+1}(p)=C_m(p)$, and this $C_m(p)$ is the equivalence class of $p$ under $\sim$, which we will denote by $C(p)$.

Let

\begin{eqnarray*}S_2&=&\{xxx,
xxyxy^R,xxyx^Ry,
xxyxyy,xxyx^Ry^R,xxyyx, xxyyx^R,  
xx^R,
xyxxy,\\
&&xyx^Rx^Ry,xyxyx,xyxyx^R, 
xyx^Ryx,xyx^Ry^Rx,
xyyx^R,
 xyxy^Rx, xyxy^Rx^R
\}\end{eqnarray*}

and

$$S_3=\{xx, xyxy, xyxy^R, xyx^Ry^R\}.$$

One checks that $s=\ell(s)$ for all $s\in S_2\cup S_3$.
The following theorems are proved in Sections \ref{2-constructions} and \ref{3-constructions}, respectively.
\begin{theorem}\label{S_2}
The patterns of $S_2$ are 2-avoidable.\end{theorem}

\begin{theorem}\label{S_3}
The patterns of $S_3$ are 3-avoidable.\end{theorem}

We will prove the following:

\begin{theorem}[Main Theorem]\label{main theorem}
Let $p$ be a binary pattern with reversal. The avoidability index of $p$ is 2, 3 or $\infty$.
\end{theorem}
In fact, we characterize exactly which of these patterns are 2-avoidable, 3-avoidable and unavoidable in the next two theorems.

\begin{theorem}\label{unavoidable}
Let $p$ be a binary pattern with reversal. If $\ell(p)$ is a prefix of one of $xyx$ and $xyx^R$, then $p$ is unavoidable; otherwise $p$ is 3-avoidable.
\end{theorem}
\begin{proof}To begin with, we note that $xyx$ and $xyx^R$ are unavoidable: If positive integer $k$ is fixed, consider any word $w$ over $T_k$ of length $2k+1$. Some letter  $a\in T_k$ appears in $w$ at least 3 times, and $w$ has a factor $aba$ where $|b|_a\ge 1$. Consider the morphism respecting reversal where $f(x)=a$, $f(y)=b$. Then $f(xyx)=f(xyx^R)=aba$, since $a=a^R$. Thus $w$ contains instances of $xyx$ and $xyx^R$; since $w$ was an arbitrary word over $T_k$, patterns $xyx$ and $xyx^R$ are not $k$-avoidable. Since $k$ was arbitrary,  they are unavoidable. A fortiori, their prefixes are unavoidable.

Now suppose that $p$ is 3-unavoidable.Without loss of generality, replace $p$ by $\ell(p)$. The first letter of $p$ is thus $x$. If $|p|=1$ we are done.
By Theorems~\ref{S_2} and \ref{S_3}, no factor of $p$ is equivalent to $xx$ or $xx^R$; the two-letter prefix of $p$ is thus $xy$ or $xy^R$. Since $p=\ell(p)$, it follows that $xy$ is a prefix of $p$. Therefore, if $|p|=2$, we are done. Since $yy$ and $yy^R$ are equivalent to $xx$ and $xx^R$ respectively, the third letter of $p$ must be $x$ or $x^R$, and one of $xyx$ and $xyx^R$ is a prefix of $p$. If $|p|\le 3$, we are done. If $|p|\ge 4$, then the fourth letter of $p$ must be $y$ or $y^R$; otherwise $p$ ends in a word equivalent to $xx$ or $xx^R$. Now, however, the length 4 prefix of $p$ is one of $xyxy$, $xyxy^R$, $xyx^Ry$ and $xyx^Ry^R$. However, $xyx^Ry$ cannot be a prefix of $p$, since $\ell(xyx^Ry)=xyxy^R$ which is 3-avoidable by Theorem~\ref{S_3}. The other possibilities are also 3-avoidable by Theorem~\ref{S_3}. We conclude that $|p|\le 3$, and our proof is complete.
\end{proof}

\begin{theorem}\label{2-avoidable}
Let $p$ be a binary pattern with reversal. Then $p$ is 2-avoidable if and only if $\ell(u)\in S_2$ for some factor $u$ of $p$.
\end{theorem}

\begin{proof}
By Theorem~\ref{S_2}, if $\ell(u)\in S_2$ for some factor $u$ of $p$, then $\ell(u)$, hence $u$, hence $p$ is 2-avoidable. 

In the other direction, suppose that for all factors $u$ of $p$, $\ell(u)\not\in S_2$. We show that $p$ is 2-unavoidable.
For each non-negative integer $n$, let $A_n$ be defined by
$$A_n=\{q:|q|=n, q = \ell(q), \mbox{ and if $u$ is a factor of $q$ then }\ell(u)\notin S_2\}.$$
If $q$ is in $A_n$, $n>0$, write $q'$ for the prefix of $q$ of length $n-1$. Then $\ell(q')\in A_{n-1}$. Thus, $q=ra$, where $r\in C(\hat{r})$, some $\hat{r}\in A_{n-1}$, $a\in\Sigma$. This allows us to compute the $A_n$:

\begin{eqnarray*}
A_0&=&\{\epsilon\}\\
A_1&=&\{x\}\\
A_2&=&\{xx, xy\}\\
A_3&=&\{xxy,xyx, xyx^R\}\\
A_4&=&\{xxyx,xxyx^R,xxyy, xyxy,xyxy^R,xyx^Ry^R,xyyx\}\\
A_5&=&\{xxyxx,xxyxy,xxyx^Rx^R\}\\
A_6&=&\phi.
\end{eqnarray*}
It follows that $A_n=\phi$, $n\ge 6.$

We have $\ell(p)\in A_{|p|}\subseteq \bigcup_{i=0}^\infty A_i=\bigcup_{i=0}^5 A_i.$ A backtracking algorithm shows that elements of $\bigcup_{i=0}^5 A_i$ are all 2-unavoidable. It follows that $p$ is 2-unavoidable.
\end{proof}

\section{Binary patterns with reversal that are 3-avoidable}\label{3-constructions}
In this section we will prove Theorem~\ref{S_3}. A {\bf square} is an instance of $xx$. It was shown by Thue \cite{thue} that squares are 3-avoidable. Any instance of $xyxy$ is necessarily a square. Therefore, both $xx$ and $xyxy$ are 3-avoidable. To prove Theorem~\ref{S_3}, it thus remains to show that $xyxy^R$ and $xyx^Ry^R$ are 3-avoidable.

Fraenkel and Simpson \cite{FS} constructed a binary sequence containing no squares other than 00, 11 and 0101. We will refer to this sequence as ${\bf f}$. 
\begin{theorem} Patterns $xyxy^R$ and $xyx^Ry^R$ are 3-avoidable.
\end{theorem}
\begin{proof}
From ${\bf{f}}$, create a word ${\bold{g}}$ by replacing each 
factor $10$ of ${\bf f}$ by $12220$. 
Word ${\bold{g}}$ has the form ${\bf g}=0^{a_1} 1^{a_2} 2^3 0^{a_3} 1^{a_4} 2^3\cdots$ where for each $i$, $1\le a_i\le 3$, since neither of $0^4 =(00)^2$ and $1^4=(11)^2$ can be a factor of ${\bf f}$.  In particular, ${\bold{g}}$ has no length 2 factor $cd$ 
where $c \equiv d + 1 $ (mod~3). Note also that word ${\bold{f}}$ never contains 
$1010$,  so that ${\bf g}$ never contains $220122201$ or $012220122$.

Suppose that $xyxy^R$, (resp., $xyx^Ry^R$) is a factor of ${\bold{g}}$. Then so is $xyxy$: Any factor $z$ of ${\bold{g}}$ containing distinct letters has a factor 
$dc$ where $d \equiv c + 1 $ (mod~3); thus $z^R$ has a length 2 factor $cd$ 
where $c \equiv d + 1$  (mod~3), so that $z^R$  cannot be a factor of ${\bold{g}}$. Since  both $y$ and $y^R$ (resp., $x$, $x^R$, $y$ and $y^R$) are factors of ${\bold{g}}$, then $y$ (resp., $x$, $y$) must be a power of a single letter, so that $y = y^R$ (resp., $x=x^R$, $y=y^R$).

Thus ${\bold{g}}$ has a factor $xyxy$, which is equivalent to having a factor 
$xx$ with $|x|\geq 2$. We show that ${\bold{g}}$ has no such factor: Suppose ${\bold{g}}$ has factor $xx$ with $|x|\geq 2$. Word $x$ must contain $2$ distinct letters, otherwise 
$xx$ consists of a letter repeated four or more times, contradicting $a_i \leq 3$. This implies that all three of $0, 1, 2$ appear in $xx$. Deleting $2$'s from 
$xx$ leaves a square over $\{0,1\}$ containing both $0$ and $1$. This must be $0101$. Then, adding the $2$'s back in, $xx$ is a factor of $201222012$; however 
the only square factor of $201222012$  is $22$, and $|x|\geq 2$. This is a contradiction. 

In conclusion, $xyxy^R$, and $xyx^Ry^R$ 
are avoided by ${\bold{g}}$, and are thus $3$-avoidable. 
\end{proof}

\section{Binary patterns with reversal that are 2-avoidable}\label{2-constructions}

In this section we will prove Theorem~\ref{S_2} using several new constructions as well as some known results. We partition $S_2$ into pieces according to the constructions used: $S_2=\bigcup_{i=1}^4 S_{2,i}$ where
\begin{eqnarray*}
S_{2,1}&=&\{xxx,xxyxyy, xxyyx,xyxxy,xyxyx\}\\
S_{2,2}&=&\{xyxyx^R\}\\
S_{2,3}&=&\{ xxyxy^R,xxyx^Ry,xxyx^Ry^R,xxyyx^R,xx^R,xyx^Rx^Ry,xyyx^R \}\\
S_{2,4}&=&\{xyxy^Rx^R,xyx^Ry^Rx,xyxy^Rx,xyx^Ryx\}.
\end{eqnarray*}
\begin{theorem}
The words of $S_{2,1}$ are 2-avoidable.
\end{theorem}
\begin{proof}These patterns, which are ordinary binary patterns, i.e., words over $\{x,y\}$, were shown to be 2-avoidable by Thue\cite {thue}, Roth\cite{roth} and Cassaigne \cite{cass}.
\end{proof}

\begin{theorem}
The sequence ${\bf f}$ of Fraenkel and Simpson avoids $xyxyx^R$.
\end{theorem}
\begin{proof}
Suppose $XYXYX^R$ is an instance of $xyxyx^R$ in ${\bf f}$, where $X$, $Y\in T_2^+$. It follows that $XYXY$ is a square of length at least 4; as there is only one such square in ${\bf f}$, this forces $X=0$, $Y=1$. However, in this case ${\bf f}$ contains the factor $YXYX^R=1010$, which is impossible.
\end{proof}

To prove Theorem~\ref{S_2}, it remains to show that the patterns of $S_{2,3}$ and $S_{2,4}$ are 2-avoidable. We do this in Sections~\ref{S_{2,3}} and \ref{S_{2,4}}, respectively.

\subsection{Patterns in $S_{2,3}$ are 2-avoidable.}\label{S_{2,3}}
We use here elementary notions of graph theory; in particular, a graph has a 2-colouring if and only if it has no odd cycles. A standard reference is by Wilson \cite{wils}. Let $p$ be a binary pattern with reversal. We use the notation
$$ a^R=\begin{cases}
x&\mbox{ if }a=x^R\\
y&\mbox{ if }a=y^R
\end{cases}.$$
 Define $G(p)$ to be the graph with vertex set $\Sigma$, and an edge between $a^R$ and $b$ whenever $ab$ is a length two factor of $p$. \vspace{.1in}

\noindent{\bf Example: }If $p= x^Rxyx^Rx^Ry$, then the length two factors are $x^Rx$, $xy$, $yx^R$, $x^Rx^R$ and $x^Ry$, giving rise to edges $xx$, $x^Ry$, $y^Rx^R$, $xx^R$ and $xy$. The graph $G(p)$ is shown in Figure~1. This graph contains odd cycles, for example, $x$--$x$, of length 1, and $x$--$x^R$--$y$--$x$, of length 3.\vspace{.1in}
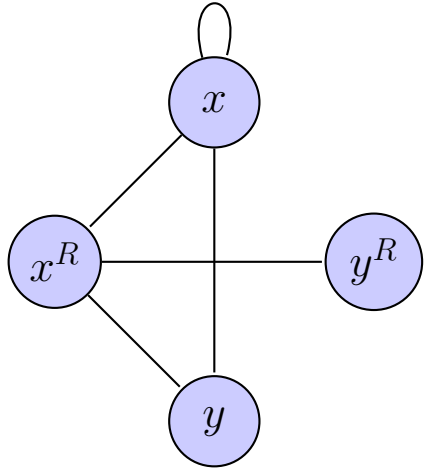
\begin{figure}
\caption{The graph $G(p)$, where $p=x^Rxyx^Rx^Ry$.}
\begin{tikzpicture}[-,shorten >=1pt,auto,node distance=3cm,
  thick,main node/.style={circle,fill=blue!20,draw,font=\sffamily\Large\bfseries},minimum size=1.2cm,every loop/.style={}]

  \node[main node] (1) {$x$};
  \node[main node] (2) [below left of=1] {$x^R$};
  \node[main node] (3) [below right of=2] {$y$};
  \node[main node] (4) [below right of=1] {$y^R$};

  \path[every node/.style={font=\sffamily\small}]
    (1)  edge  node[left] {} (2)
        edge node[left] {} (3)
        edge [loop above] node {} (1)
    (2)  edge node {} (4)
       edge node[left] {} (3);
 
\end{tikzpicture}
\end{figure}

\begin{theorem}\label{alternating}
Let $p\in\{x,x^R,y,y^R\}^*$. An instance of $p$ appears in $(01)^\omega$ if and only if $G(p)$ is bipartite.
\end{theorem}
\begin{proof} Let $u$ and $v$ be factors of $(01)^\omega$. Then $uv$ is a factor of $(01)^\omega$ exactly when $u^R$ and $v$ begin with different letters.
Suppose $G(p)$ is bipartite, and let $c:G(p)\rightarrow\{0,1\}$ be a legal colouring. Now let $X$ be the shortest string beginning with $c(x)$ and ending with $c(x^R)$; thus $X$ is a factor of $(01)^\omega$ with $1\le |X|\le 2$. Similarly, let $Y$ be the shortest string beginning with $c(y)$ and ending with $c(y^R)$. 

 Define the morphism $h:\{x,y\}\rightarrow\{0,1\}^*$ by $h(x)=X$, $h(y)=Y$. If $a\in\{x,x^R,y,y^R\}$, then $h(a)$ begins with  $c(a)$ and ends in $c(a^R)$.
Suppose $ab$ is a length two factor of $p$. Then $a^Rb$ is an edge of $G(p)$, and $c(a^R)\ne c(b)$ so that $(h(a))^R$ and $h(b)$ begin with different letters. It follows that $h(ab)$ is a factor of $(01)^\omega$. By induction, we see that $h(p)$ is a factor of $(01)^\omega$.

In the other direction, suppose that $h:\{x,y\}\rightarrow\{0,1\}^*$ is a morphism such that $h(p)$ is a factor of $(01)^\omega$. For $a\in\{x,x^R,y,y^R\}$, define the 2-colouring $c$ by choosing $c(a)$ to be the first letter of $h(a)$. If this is not a legal colouring, then for some letters $a,b\in\{x,x^R,y,y^R\}$, there is an edge $ab$ in $G(p)$ with $c(a)=c(b)$. This implies that $a^Rb$ is a factor of $p$, but $h(a)$ and $h(b)$ start with the same letter. Then $h(a^R)$ ends with the same letter that begins $h(b)$, forcing 00 or 11 to be a factor of $h(p)$, which is in turn a factor of $(01)^\omega$. This is impossible.
\end{proof}
 \begin{corollary} Every pattern in $S_{2,3}$  is avoided by $(01)^\omega$.
\end{corollary}
\begin{proof} Graph $G(xx^R)$ contains the loop $x^R$--$x^R$, i.e., a 1-cycle. For each of the other patterns $p\in S_{2,2}$, $G(p)$ contains a triangle.
\end{proof}
\subsection{The patterns of $S_{2,4}$ are 2-avoidable.}\label{S_{2,4}}

The Thue-Morse word is the fixed point ${\bf t}=h^\omega(0)$, of the binary morphism $h$ given by $h(0)=01$, $h(1)=10$. 
Thue \cite{thue} showed that ${\bf t}$ avoids {\bf overlaps}, i.e., instances of $xxx$ or $xyxyx$.

Suppose $f$ is any non-erasing binary morphism such that $f(0)=0$. Let ${\bf w}=f({\bf t})$.
\begin{lemma}\label{bounding search}
Let $u$ be a factor of ${\bf w}$. Then $u$ is a factor of $f(v)$, for some factor $v$ of ${\bf t}$ with $$|v|\le {2\over |f(1)|+1}(|u|+3|f(1)|-3).$$
\end{lemma}
\begin{proof}
For some words $p$ and $s$ such that $|p|,|s|\le |f(1)|-1$, we have $sup=f(v)$, where $v$ is some factor of ${\bf t}$. We can write $v=s_1h(v_1)p_1$, where $|p_1|,|s_1|\le 1$ and $v_1$ is a factor of ${\bf t}$. Since $|h(v_1)|_1=|h(v_1)|/2$, it follows that 
$$|v|_1\ge {|v|-2\over 2}.$$
Thus 
\begin{eqnarray*}
|u|&\ge& |f(v)|-2(|f(1)|-1)\\
&=&|f(1)||v|_1+|v|_0-2|f(1)|+2\\
&=&|f(1)||v|_1+(|v|-|v|_1)-2|f(1)|+2\\
&=&(|f(1)|-1)|v|_1+|v|-2|f(1)|+2\\
&\ge& (|f(1)|-1)\left({|v|\over 2}-1\right)+|v|-2|f(1)|+2\\
&=& (|f(1)|+1){|v|\over 2}-3|f(1)|+3.
\end{eqnarray*}
Thus $$|v|\le {2\over |f(1)|+1}(|u|+3|f(1)|-3).$$
\end{proof}
\begin{lemma}Let $v$ be a factor of ${\bf t}$ of odd length. Then $v$ is a factor of $h(v')$ for some factor $v'$ of ${\bf t}$ of length $(|v|+1)/2$.  
\end{lemma}
\begin{proof}Omitted.
\end{proof}
\begin{corollary}\label{t bound}
Every factor of ${\bf t}$ of length $2^n+1$ is a factor of the prefix of ${\bf t}$ of length $7(2^n).$
\end{corollary}
\begin{proof}
All length two binary words are factors of $0110100$, the length 7 prefix of ${\bf t}$. The result follows by applying the previous lemma $n$ times.\end{proof}

\subsubsection{Patterns $xyxy^Rx^R$ and  $xyx^Ry^Rx$ are  2-avoidable.}

Let $f_1$ be the binary morphism given by $f_1(0)=0$, $f_1(1)=00101101111$, and let ${\bf w_1}=f_1({\bf t})$.

\begin{theorem}
The sequence ${\bold{w_1}}$ avoids $xyxy^Rx^R$.
\end{theorem}
\begin{proof}
Let $z$ be a factor of ${\bf w_1}$ such that $z^R$ is also a factor of ${\bf w_1}$.\vspace{.1in}
We claim that  $|z|\le 6$.
 Otherwise, replacing $z$ by its length 7 prefix, ${\bf w_1}$ has a length 7 factor $z$ such that $z^R$ is also a factor of ${\bf w_1}$. we note that $|f_1(1)|=11$, so that by Lemma~\ref{bounding search}, $z$ is a factor of $f_1(v)$, some factor $v$ of ${\bf t}$ where 
$$|v|\le\frac{2}{11+1}(7+3(11)-3)<7.$$
Certainly then an extension of $v$ is a factor of ${\bf t}$ of length $9=2^3+1$, so that by Corollary~\ref{t bound}, $v$ is a factor of the prefix of ${\bf t}$ of length 56. This implies that $z$ and $z^R$ are factors of $f_1(\tau)$, where $\tau$ is the prefix of ${\bf t}$ of length 56. A search shows that $f_1(\tau)$ has no
length 7 factor $z$ such that $z^R$ is also a factor of $f_1(\tau)$.

Suppose that $XYXY^RX^R$ is a factor of ${\bf w_1}$ with $X$, $Y\ne\epsilon$. Since both $X$ and $X^R$, and both $Y$ and $Y^R$ are factors of ${\bf w_1}$, it follows that $|X|$, $|Y|\le 6$, and $|XYXY^RX^R|\le 30$.
By Lemma~\ref{bounding search}, $XYXY^RX^R$ is a factor of $f_1(v')$, some factor $v'$ of ${\bf t}$ where 
$$|v'|\le\frac{2}{11+1}(30+3(11)-3)=10<2^4+1.$$
By Corollary~\ref{t bound}, $v'$ is a factor of the prefix of ${\bf t}$ of length 112, so that $XYXY^RX^R$ is a factor of $f_1(\tau')$, where $\tau'$ is the prefix of ${\bf t}$ of length 112. However, a search shows that $f_1(\tau')$ has no factor $XYXY^RX^R$ with $|X|,|Y|\le 6$.

We conclude that ${\bf w_1}$ avoids $xyxy^Rx^R$.
\end{proof}

Next, we give an infinite binary word that avoids the pattern $xyx^Ry^Rx$. Let $f_2$ be the binary morphism with $f_2(0) = 0$, and $f_2(1) = 00101111$. Let ${\bold{w_2}}=f_2({\bf t})$.

\begin{theorem}
The sequence ${\bold{w_2}}$ avoids $xyx^Ry^Rx$.
\end{theorem}
\begin{proof}
Suppose $XYX^RY^RX$ is a factor of ${\bf w_2}$, $X,Y\ne\epsilon$. As in the proof of the previous lemma, we find that 
$|X|$, $|Y|\le 6$, so that  $XYX^RY^RX$ is a factor of $f_2(\tau')$, where $\tau'$ is the prefix of ${\bf t}$ of length 112. However, a search shows that $f_2(\tau')$ has no such factor. 
\end{proof}

\subsubsection{Avoiding $xyxy^Rx$}
Let $f_3$ be the binary morphism given by $f_3(0)=0$, $f_3(1)=001011$. Let ${\bf w_3}=f_3({\bf t})$.
Define $$\Upsilon=\{1,0,11, 10, 00, 01, 010, 011, 001, 000, 110,
100, 101,0110,$$ $$0000,0001,1001,1000,00001,10000,10001,100001\}.$$

\begin{lemma}
If both $Y$ and $Y^R$ are factors of ${\bf w_3}$, then $Y\in\Upsilon.$
\end{lemma}
\begin{proof} This is established by finite search, using Lemma~\ref{bounding search} and Corollary~\ref{t bound}.
\end{proof}

\begin{lemma}\label{upsilon}Suppose that $X$ and $Y$ are words such that $XY$, $XY^R$, $YX$ and $YX^R$ are all factors of ${\bf w_3}$, and $|X|\ge 3$. Then $Y\in\{0,1,00\}$.
\end{lemma}
\begin{proof} By the previous lemma, $Y\in\Upsilon$. Define
 $${\mathcal X}_1(Y)=\{\chi:|\chi|=3\mbox{ and }\chi Y, \chi Y^R\mbox{ are both factors of }{\bf w_3}\},$$
 $${\mathcal X}_2(Y)=\{\chi:|\chi|=3\mbox{ and }Y\chi,Y^R \chi\mbox{ are both factors of }{\bf w_3}\}.$$ 
The length 3 suffix of $X$ must be in ${\mathcal X}_1$, and the length 3 prefix of $X$ must be in ${\mathcal X}_2$.
 We can compute ${\mathcal X}_1(Y)$ and ${\mathcal X}_2(Y)$ by a finite search, using Lemma~\ref{bounding search} and Corollary~\ref{t bound}. For $Y\in\Upsilon-\{0,1,00\}$, we find that
${\mathcal X}_1(Y)=\phi$ or ${\mathcal X}_2(Y)=\phi.$ The result follows.
\end{proof}

Let $u$ be a factor of ${\bf w_3}$. Define a {\bf left completion} of $u$ to be a word $v=f_3(t)$, such that $u$ is a suffix of $f_3(t)$, but $u$ is a not a suffix of any proper suffix of $v$ of the form $f_3(t')$. Thus, for example, $001011$ is a left completion of 11 and of 011, but 01 has no left completion.

\begin{lemma}
Let $u$ be a factor of ${\bf w_3}$ which ends in 11. Then $u$ has a unique left completion.
\end{lemma}
\begin{proof}
Induction.
\end{proof}

We remark that if $p$ is a prefix of ${\bf w_3}$ and 11 is a suffix of $p$, then $p=f_3(t)$ for some prefix $t$ of ${\bf t}$.

\begin{theorem}
Word ${\bf w_3}$ avoids $xyxy^Rx$.
\end{theorem}
\begin{proof}
A finite search shows that ${\bf w_3}$ contains no factor $XYXY^RX$ with $1\le |X|\le 8$ and $1\le |Y|\le 2$. Suppose that, nevertheless, ${\bf w_3}$ contains some factor $XYXY^RX$, $|X|,|Y|\ne 0$. By Lemma~\ref{upsilon}, $Y\in\{0,00,1\}$, so that $|X|\ge 9$. This means that $Y=Y^R$, so it will be notationally simpler to write $XYXYX$ for $XYXY^RX$. It is easy to show that (or alternatively, by a finite search, invoking Lemma~\ref{bounding search} and Corollary~\ref{t bound}) any factor $\chi$ of ${\bf w_3}$ with $|\chi|=9$ contains the factor 11. Therefore, write $X=X'X^{\prime\prime}$, where 11 is a suffix of $X'$, and $|X^{\prime\prime}|_{11}=0$.

Let $pXYXYX=pX'X^{\prime\prime}YX'X^{\prime\prime}YX'X^{\prime\prime}$ be a prefix of ${\bf w_3}$ for some $p$. Let $v$ be the left completion of $X'$. 
Then $pX'X^{\prime\prime}YX'=f_3(t_1)$ and $pX'X^{\prime\prime}YX'X^{\prime\prime}YX'=f_3(t_2)$ for some prefixes $t_1$ and $t_2$ of ${\bf t}$. It follows that $X^{\prime\prime}YX'=f_3(t_3)$ for the factor $t_3=t_1t_2^{-1}$ of ${\bf t}$. Therefore, some suffix of $X^{\prime\prime}YX'$ is a left completion of $X'$; by uniqueness of left completions,  $v$ is a suffix of $X^{\prime\prime}YX'$. Write $v=f_3(t_4)$ for some factor $t_4$ of ${\bf t}$. Then $X^{\prime\prime}YX'=f_3(t_5t_4)$, where $t_5=t_3t_4^{-1}$.
Since $X'$ has a unique left completion, $v$ is a suffix of $pX'$. Write $pX'=p'v$, and $pXYXYX'=p'vX^{\prime\prime}YX'X^{\prime\prime}YX'=p'f_3(t_4t_5t_4t_5t_4)$.
 Since $f_3$ is injective, 
the factor $f_3(t_4t_5t_4t_5t_4)$ of ${\bf w_3}$ implies the existence of the overlap $t_4t_5t_4t_5t_4$ in ${\bf t}$, contradicting the overlap-freeness of ${\bf t}$.

\end{proof}

\subsubsection{Pattern $xyx^Ryx$ is 2-avoidable.}

Let $f_4$ be the binary morphism given by $f_4(0)=0$, $f_4(1)=1000010011$. Let ${\bf w_4}=f_4({\bf t})$. We observe that 011 only occurs in ${\bf w_4}$ as a suffix of $f_4(1)$. More formally:
\begin{lemma}If $p011$ is a prefix of ${\bf w_4}$, then $p011=f_4(t1)$ for some prefix $t1$ of ${\bf t}$.
\end{lemma}

\begin{corollary}\label{1 in t}Let $py$ be a prefix of ${\bf w_4}$. Suppose that $y=f_4(\hat{t})$ for some factor $\hat{t}$ of ${\bf t}$ where $|\hat{t}|_1>0$. Then $p=f_4(\tau)$ for some prefix $\tau$ of ${\bf t}$.
\end{corollary}
\begin{proof} Let $t'$ be the shortest prefix of $\hat{t}$ that contains a 1. Thus $t'=0^n1$, some $n\in\{0,1,2\}$, and $pf_4(t')=p0^nf_4(1)$ is a prefix of ${w_4}$. By the previous lemma, $p0^nf_4(1)=f_4(t1)$ for some prefix $t1$ of ${\bf t}$. Then $0^n$ is a suffix of $t$, and  the result follows letting $\tau=t0^{-n}$.
\end{proof}\begin{corollary}Let $py$ be a prefix of ${\bf w_4}$. Suppose that $y=f_4(\hat{t})$ for some factor $\hat{t}$ of ${\bf t}$, and $|y|\ge 6$.  Then $p=f_4(\tau)$ for some prefix $\tau$ of ${\bf t}$.
\end{corollary}
\begin{proof}The only factors of ${\bf t}$ not containing a 1 are $\epsilon$, 0 and 00. Since $|y|\ge 6>|f_4(00)|=2$, we conclude that $|\hat{t}|_1>0$. 
\end{proof}
Recall that a factor $y$ of ${\bf w_4}$ is {\bf bispecial} if $0y$, $1y$, $y0$, $y1$ are all factors of ${\bf w_4}$.
\begin{lemma}
Let $y$ be a bispecial factor of ${\bf w_4}$ with $|y|\ge 6$. Then $y=f_4(t)$ for some factor $t$ of ${\bf t}$.
\end{lemma} 
\begin{proof}
Either $y$ is an internal factor of $f_4(1)$, or $y$ can be written as $y=sf_4(t)p$ where $t$ is a factor of ${\bf t}$, $s$ is a suffix of $f_4(1)$, $p$ is a prefix of $f_4(1)$, and $|s|,|p|<|f_4(1)|$. Now the internal factors of $f_4(1)$ of length at least 6 are 000010, 000100, 001001, 0000100, 0001001 and 00001001. Each of these only occurs in ${\bf w_4}$ inside a copy of $f_4(1)$; therefore, none of these are bispecial (or even right special or left special).

Therefore write $y=sf_4(t)p$ where $s$ is a suffix of $f_4(1)$ and $p$ is a prefix of $f_4(1)$, and $|s|,|p|<|f_4(1)|$. 

Suppose $s=1$. Word $t$ begins with 00, 01 or 1, or else $t=\epsilon$. Thus one of $1|0|0|100$, $1|0|1000$ or $1|10000$ is a prefix of $y$; here the vertical bars mark the divisions in ${\bf w_4}$ between $f_4$-images of letters. In each case, we see that $y$ must be preceded by 1 in ${\bf w_4}$, contradicting the assumption that $y$ is bispecial. If $s=11$, then one of  $11|0|0|10$, $11|0|100$ or $11|1000$ is a prefix of $y$. In each case $y$ is always preceded by 0, again contradicting the assumption that $y$ is bispecial. If $|s|\ge 3$, then $s$ ends in 011; however, 011 only arises in ${\bf w_4}$ as a suffix of $f_4(1)$, so that the letter preceding $s$ (and thus $y$) in ${\bf w_4}$ must always be the letter preceding $s$ in $f_4(1)$. The cases where $|s|>0$ therefore all lead to a contradiction. We conclude that $s=\epsilon.$

If $p=1$, then one of $011|0|0|1$, $0011|0|1$ and $10011|1$ is a suffix of $y$. This implies that $y$ is always followed in ${\bf w_4}$ by $0$, a contradiction, since $y$ is bispecial. If $p=10$, then one of $11|0|0|10$, $011|0|10$ and $0011|10$ is a suffix of $y$, and $y$ is always followed by $0$ in ${\bf w_4}$. If $p=100$, then $y$ has a suffix $1|0|0|100$, $11|0|100$ or $011|100$, and $y$ is always followed by a 0. If $|p|\ge 4$, then $p$ begins 1000, which only arises in ${\bf w_4}$ as a prefix of $f_4(1)$, so that the letter following $p$ (and thus $y$) in ${\bf w_4}$ must always be the letter following $p$ in $f_4(1)$. We conclude that $p=\epsilon$.

Since $p=s=\epsilon$, $y=f_4(t)$ for some factor $t$ of ${\bf t}$, as claimed.
\end{proof}
\begin{corollary}\label{bispecial} Let $y$ be a bispecial factor of ${\bf w_4}$ with $|y|\ge 6$. Let $py$ be a prefix of ${\bf w_4}$. Then $p=f_4(\tau)$ for some prefix $\tau$ of ${\bf t}$ and $y=f_4(t)$ for some factor $t$ of ${\bf t}$.
\end{corollary}
\begin{theorem} The word ${\bf w_4}$ avoids $xyx^Ryx$. 
\end{theorem}
\begin{proof}
Suppose not. Let $u$ be a factor of ${\bf w_4}$ of the form $ u=XYX^RYX$, with $X$, $Y\ne \epsilon$, and such that $u$ is as short as possible.

Both $X$ and $X^R$ are factors of ${\bf w_4}$. By Lemma~\ref{bounding search}, each length 21 factor of ${\bf w_4}$ will be a factor of $f_4(v)$, for some appropriate length 8 factor $v$ of ${\bf t}$. By Corollary~\ref{t bound}, every length 8 factor of ${\bf t}$ appears in the length 56 prefix of ${\bf t}$. We can therefore effectively list all length 21 factors of ${\bf w_4}$. One verifies that if $z$ is a length 21 factor of ${\bf w_4}$, then $z^R$ is {\bf not} a factor. Thus, since both $X$ and $X^R$ are factors of ${\bf w_4}$, $|X|\le 20$.\vspace{.1in}

\noindent{\bf Subcase 1: $|Y|\le 5$.} In this case, $|XYX^RYX|\le 70$, and by Lemma~\ref{bounding search}, $XYX^RYX$ is a factor of $f_4(v)$, for some factor $v$ of ${\bf t}$ of length  17. The length 17 factors of ${\bf t}$ all lie in the length 112 suffix $h^4(0110100)$ of ${\bf t}$, and a finite search shows that no factor $XYX^RYX$ occurs in $f_4(h^4(0110100))$. This case therefore cannot occur.\vspace{.1in}

\noindent{\bf Subcase 2: $|Y|\ge 6$.}

\noindent {\bf Subcase 2a: The first and last letters of $X$ are different.} In this case, write $X=aX'b$ where $a,b\in\{0,1\}$, $a\ne b$. Then $XYX^RYX=aX'bYb(X')^RaYaX'b$, and we see that  $Y$ is bispecial. Let $\pi XYX^RYX$ be a prefix of ${\bf w_4}$. Applying Corollary~\ref{bispecial} several times, we see that $\pi X=f_4(t_0)$, $\pi XY=f_4(t_1)$, $\pi XYX^R=f_4(t_2)$ and $\pi XYX^RY=f_4(t_3)$ for some prefixes $t_0, t_1, t_2$ and $t_3$ of ${\bf t}$. It follows that $X^R=f_4(t)$, where $t$ is the factor $(t_1)^{-1}t_2$ of ${\bf t}$. If the last letter of $X$ is a 1, then $\pi X=f_4(t_0)$ must have suffix 11; this implies that 11 is a prefix of $X^R=f_4(t)$, which is impossible. However, if the last letter of $X$ is a 0, then the first letter of $X$ is a 1. Thus the last letter of $X^R$ is a 1, and $11$ is a suffix of $X^R=f_4(t)$. Then 11 is a prefix  of $X$, and hence of $(\pi XYX^RY)^{-1}{\bf w_4}=f_4(t_3^{-1}{\bf t})$, which is also impossible. \vspace{.1in}

\noindent {\bf Subcase 2b: The first and last letters of $X$ are the same.}
Write $X=a\chi a$ where $a\in\{0,1\}$. Then $u=a\chi aYa\chi^RaYa\chi a$ has the proper factor $\chi\Upsilon\chi^R\Upsilon\chi$, where $\Upsilon=aYa$. If $\chi\ne\epsilon$, we have a contradiction, since $u$ was to be as short as possible. We conclude that $\chi=\epsilon$, and $X\in\{0,00,1,11\}$, whence $X=X^R$.

Since the finite search of Subcase 1 shows that we must have $|XYX^RYX|>70$, we may assume that $|Y|>(70-3|X|)/2\ge 32.$ Therefore,
write $Y=sf_4(t)p$ where $t$ is a factor of ${\bf t}$, $s$ is a suffix of $f_4(1)$, $p$ is a prefix of $f_4(1)$, and $|s|,|p|<|f_4(1)|$. It follows that $|f_4(t)|\ge 32 - 18=14$. We conclude that $|t|_1>0$. 

We now have $XYX^RYX=Xsf_4(t)pXsf_4(t)pX,$ where $1\le |X|\le 2$. Write 
$${\bf w_4}=\pi Xsf_4(t)pXsf_4(t)pX{\bf \sigma}.$$
 By Corollary~\ref{1 in t},
$$\pi Xs=f_4(t_0),
\pi Xsf_4(t)=f_4(t_1),
\pi Xsf_4(t)pXs=f_4(t_2),
\pi Xsf_4(t)pXsf_4(t)=f_4(t_3),
$$ 
for some prefixes $t_0$, $t_1$, $t_2$, $t_3$ of ${\bf t}$. Therefore, $pXs=f_4((t_1)^{-1}t_2)$.\vspace{.1in}

\noindent {\bf Subcase 2bi: The first and last letters of $pXs$ are both 0.} Since $p$ and $s$ are, respectively, a prefix and suffix of $f_4(1)$, which begins and ends with 1, this forces $p=s=\epsilon$, $X=0^n$, $n\in\{1,2\}$.  Thus
$${\bf w_4}=\pi 0^nf_4(t)0^nf_4(t)0^n{\bf \sigma}.$$
 By Corollary~\ref{1 in t},
$\pi 0^nf_4(t)0^nf_4(t)0^n=f_4(\tau)$, some prefix $\tau$ of ${\bf t}$, so that ${\bf t}$ contains the overlap $0^nt0^nt0^n$. This is impossible.\vspace{.1in}

\noindent {\bf Subcase 2bii: The first letter of $pXs$ is a 0, and the last letter is a 1.}
This forces $p=\epsilon$, since otherwise $p$, and hence $pXs$, starts with a 1. Then $X$ starts with a 0, so $X=0^n$, $n\in\{1,2\}$. Since $pXs=f_4((t_1)^{-1}t_2)$ ends in a 1, this 
forces $s=f_4(1)$, contradicting $|s|<|f_4(1)|$.\vspace{.1in}

\noindent {\bf Subcase 2biii: The first letter of $pXs$ is a 1, and the last letter is a 0.}
This forces $s=\epsilon$, since otherwise $s$, and hence $pXs$, ends with a 1. Then $X$ ends with a 0, so $X=0^n$, $n\in\{1,2\}$. Since $pXs=f_4((t_1)^{-1}t_2)$ starts with a 1, this 
forces $p=f_4(1)$, contradicting $|p|<|f_4(1)|$.\vspace{.1in}

\noindent {\bf Subcase 2biv: The first and last letters of $pXs$ are both 1.}
It follows that $pXs=f_4((t_1)^{-1}t_2)$ has $f_4(1)$ as a prefix and as a suffix. Now $|pXs|\le 9 + 2 + 9 = 20$, forcing $pXs\in\{f_4(1),f_4(11)\}$. If $pXs=f_4(1)$, then $pX$ begins with a 1, and $Xs$ ends in a 1. It follows that 
$\pi Xs=f_4(t_0)$ ends in a 1, so that 1 is the last letter of $t_0$, while $pX{\bf \sigma}=f_4(t_3^{-1}{\bf t})$, so that $t_3^{-1}{\bf t}$ begins with 1. 
Thus ${\bf t}$ contains the overlap $1t1t1$, which is impossible.

On the other hand, if $pXs=f_4(11)$, then we must have $p=100001001$, $X=11$, $s=000010011$.  Again,  
$\pi Xs=f_4(t_0)$ ends $1f_4(1)$, so that 11 is a suffix  of $t_0$; also $pX{\bf \sigma}=f_4(t_3^{-1}{\bf t})$, has prefix $f_4(1)1$, so that $(t_3)^{-1}{\bf t}$ has prefix 11. 
Thus ${\bf t}$ contains the overlap $11t11t11$, which is impossible. 
\end{proof}

\section{Conclusion/Further Discussion}

We note that in 1992, Roth \cite{roth} proved that every length six binary pattern is 2-avoidable. Our Theorem~\ref{S_2} shows that this is also true for binary patterns with reversal.

It would be nice now to perhaps see if our results could be generalized to ternary patterns or beyond. Another natural desiridatum would be an effective characterization of which patterns with reversal are avoidable.

\end{document}